\documentclass[review,onefignum,onetabnum]{siamart250211}

\usepackage{lipsum}
\usepackage{amsfonts}
\usepackage{graphicx}
\usepackage{epstopdf}
\usepackage{algorithmic}
\usepackage{amssymb}
\usepackage[english]{babel}
\usepackage{array}
\usepackage{adjustbox}
\usepackage{relsize}
\usepackage{spverbatim}
\usepackage{listings}
\usepackage{mathrsfs}
\usepackage{derivative}
\usepackage{mathtools}
\usepackage{float}
\usepackage[section]{placeins}

\usepackage{hyperref}
\usepackage{cleveref}

\usepackage{etoolbox}

\makeatletter
\newcommand{\crefcomma}[1]{%
  \begingroup
    \def\crefcomma@sep{}%
    \forcsvlist{\crefcomma@do}{#1}%
  \endgroup
}
\newcommand{\crefcomma@do}[1]{%
  \ifx\crefcomma@sep\@empty\else,~\fi
  \cref{#1}%
  \def\crefcomma@sep{,}%
}
\newcommand{\Crefcomma}[1]{%
  \begingroup
    \def\crefcomma@sep{}%
    \forcsvlist{\Crefcomma@do}{#1}%
  \endgroup
}
\newcommand{\Crefcomma@do}[1]{%
  \ifx\crefcomma@sep\@empty\else,~\fi
  \Cref{#1}%
  \def\crefcomma@sep{,}%
}
\makeatother

\DeclareMathAlphabet{\mathpzc}{OT1}{pzc}{m}{it}

\newcolumntype{L}{>{$}l<{$}}

\ifpdf
  \DeclareGraphicsExtensions{.eps,.pdf,.png,.jpg}
\else
  \DeclareGraphicsExtensions{.eps}
\fi

\newsiamremark{hypothesis}{Hypothesis}
\crefname{hypothesis}{Hypothesis}{Hypotheses}
\newsiamthm{claim}{Claim}
\newtheorem{remark}{Remark}
\nolinenumbers

\headers{The general Brannan coefficient conjecture I}{T. M. Dunster}

\title{The general Brannan coefficient conjecture I: Watson-lemma approximations}

\author{T. M. Dunster\thanks{Department of Mathematics and Statistics, San Diego State University, 5500 Campanile Drive, San Diego, CA 92182-7720, USA. 
  (\email{mdunster@sdsu.edu}, \url{https://tmdunster.sdsu.edu}).}
  }

\usepackage{amsopn}

\makeatletter
\newcommand*{\addFileDependency}[1]{
  \typeout{(#1)}
  \@addtofilelist{#1}
  \IfFileExists{#1}{}{\typeout{No file #1.}}
}
\makeatother

\nolinenumbers

\ifpdf
\hypersetup{
  pdftitle={The general Brannan coefficient conjecture I: Watson-lemma approximations},
  pdfauthor={T. M. Dunster}
}
\fi

\begin{document}

\maketitle

\begin{abstract}
The coefficients $A_n(\alpha,\beta,\omega)$ in the Maclaurin expansion $(1+\omega z)^{\alpha}(1-z)^{-\beta}= \sum_{n=0}^{\infty} A_n(\alpha,\beta,\omega)z^n$ are studied, where $\omega,z \in \mathbb{C}$ with $|z| < |\omega|=1$, and $\alpha,\beta \in (0,1]$. In 1973 D. A. Brannan conjectured that $|A_n(\alpha,\beta,\omega)|\le A_n(\alpha,\beta,1)$ for each positive odd integer $n$, and showed it is true for $n=3$. This has recently been proven for all odd integers $n\ge5$ by a number of authors in aggregate for the special case $\beta=1$. In this paper hypergeometric integral representations and Watson-type approximations are utilised, from which the general problem is reduced to numerically evaluating the minima of certain simple, explicit, slowly-varying functions over compact domains. From the positivity of these constants it is shown that the conjecture holds for $\alpha, \beta \in (0,1]$, $0 \le |\arg(\omega)| \le \pi-\phi_0$ and $n=5,7,9,\ldots$, where $\phi_0=0.061$.
\end{abstract}



\section{Introduction}

Consider the Maclaurin series of the following univalent function
\begin{equation}
\label{eq01}
\frac{(1+\omega z)^{\alpha}}{(1-z)^{\beta}}
= \sum_{n=0}^{\infty} A_n(\alpha,\beta,\omega)z^n,
\end{equation}
where $z,\omega\in\mathbb{C}$ with $|z|<1=|\omega|$, and $\alpha,\beta \in (0,1]$. Each $A_n(\alpha,\beta,\omega)$ is a polynomial in $\omega$ of degree $n$, and it is well known that
\begin{equation}
\label{eq02}
A_n(\alpha,\beta,\omega)
= \frac{(\beta)_n}{n!}\,
{}_2F_1(-n,-\alpha;\,1-\beta-n;\,-\omega),
\end{equation}
where ${}_2F_1$ is the Gauss hypergeometric function given by
\begin{equation}
\label{eq03}
{}_2F_1(a,b;\,c;\,z)
= \sum_{k=0}^{\infty} \frac{(a)_k(b)_k}{(c)_k\,k!}\,z^k,
\end{equation}
for $|z|<1$, and elsewhere by analytic continuation, with the Pochhammer symbol defined in the usual way
\begin{equation}
\label{eq04}
(a)_n = \frac{\Gamma(a+n)}{\Gamma(a)}.
\end{equation}

The following conjecture due to Brannan \cite{Brannan:1974:OCP} in 1973 arose in relation to coefficient problems involving functions of bounded boundary rotation.

\medskip\noindent
\textbf{Brannan's Conjecture.}
Suppose $\alpha,\beta\in(0,1]$ and $0 \le |\theta| \le \pi$. Then
\begin{equation}
\label{eq05}
\left|A_n(\alpha,\beta,e^{i\theta})\right|
\le A_n(\alpha,\beta,1)
\quad (n=3,5,7,\ldots).
\end{equation}

In this paper we prove this to be true for $0 \le |\theta| \le \pi-0.061$ and $n=5,7,9,\ldots$;  Brannan \cite{Brannan:1974:OCP} had already shown it to be true for $n=3$.

\begin{remark}
From \cite{Barnard:2015:BCA} it was shown for $\alpha,\beta \in(0,1)$ and positive integer $n$
\begin{multline}
\label{eq06}
A_n(\alpha,\beta,\omega)
=\frac{\Gamma(n+\beta)}{\Gamma(\beta)\,n!}
+\frac{\Gamma(n-\alpha+\beta)
\Gamma(1+\alpha)\sin(\pi\alpha)}
{\pi\Gamma(\beta)\,n!}
\\ \times
\int_{0}^{1} p_n(\omega,t)\,t^{-1-\alpha}(1-t)^{n-1+\beta}\,dt,
\end{multline}
where
\begin{equation}
\label{eq07}
p_n(\omega,t)
=1-\left(1-\frac{\omega t}{1-t}\right)^{n},
\end{equation}
and from this it is straightforward to show that $A_n(\alpha,\beta,1) >0$ if $n$ is odd; the same is not necessarily true if $n$ is even.
\end{remark}

As we mentioned, Brannan \cite{Brannan:1974:OCP} verified his conjecture for the first nontrivial odd case $n=3$, and also showed that the inequality does not hold in general for even $n$. Aharonov and Friedland \cite{Aharonov:1972:OAI} proved the inequality \cref{eq05} in the case $\alpha,\beta \ge 1$. 

For the special case $\beta=1$ and $\alpha\in (0,1)$, early proofs were obtained for the individual odd indices $n=5$ (Milcetich \cite{Milcetich:1989:OAC}), and $n=7$ (Barnard, Pearce and Wheeler \cite{Barnard:1997:OAC}). Jayatilake \cite{Jayatilake:2013:BCF} later verified the conjecture for all odd $n\le 51$ using a computer-assisted squaring procedure.

More recently, Szász \cite{Szasz:2020:OTB} and (independently) Deniz, Çaglar and Szász \cite{Deniz:2020:TFS} proved complementary large-$n$ ranges via integral representations, giving a cumulative proof of Brannan’s conjecture for $\beta=1$. 
Barnard and Richards \cite{Barnard:2021:ADP} subsequently provided a fully direct analytic proof for $\beta=1$ valid for all positive odd $n$.

In the general two-parameter setting, Ruscheweyh and Salinas \cite{Ruscheweyh:2007:OBC} established the conjecture for all odd $n$ on the diagonal $\alpha=\beta\in(0,1)$. Recently, Cot\^{\i}rla and Sz\'asz \cite{Cotirla:2024:OTG} in 2024 proved the conjecture in the region $\frac34 \le \alpha\le \beta\le 1$ for all odd indices, using a new integral representation.

The plan of this paper is as follows.  In \cref{sec:int-lower-bnds} we first express $A_n(\alpha,\beta,\omega)$ in an alternative hypergeometric form, which is shown to be expressible as a loop-integral.  This is then decomposed into a sum of two Laplace-type integrals, and these are employed to obtain a quantitative lower bound for the difference between the left- and right-hand sides of \cref{eq05}. The resulting representation yields a tractable lower bound for the difference at $\omega=1$ and, after taking moduli, for $\omega$ on the unit circle.

We then in \cref{sec:Watson} apply the method of Watson’s lemma (see, for example, \cite[Chap.~3]{Olver:1997:ASF}, \cite[Chap.~2]{Temme:2015:AMF}, \cite[Chap.~5]{Wong:1989:AAI}) to derive a  two-term large-$n$ asymptotic approximation for the lower bound, together with explicit remainders. These remainder terms are given by integral representations, and from these we obtain simple lower bounds by computing the infima of certain smooth, explicit, slowly-varying functions over compact sets.

In \cref{sec:P5} we prove the conjecture for $n=5$, with $\alpha,\beta\in(0,1]$ and $0\le \arg(\omega)\le \pi-0.061$, by establishing positivity of the lower bound constructed from the main terms and the explicit lower error bounds of \cref{sec:Watson}.

In \cref{sec:n>5} we extend the result to all odd $n\ge 7$.  This is done by proving positivity of an appropriate derivative, again reducing the required estimates to minimizing a pair of simple, explicit, slowly-varying functions over a compact domain. Finally, in \cref{sec:Meijer} we discuss a method that will be used to tackle the remaining segment $\pi-0.061<|\arg(\omega)|\le \pi$.

We emphasise that all the functions being minimized, with the exception of the real-valued gamma function $\Gamma(x)$ with $x \in (0,2]$, and the real-valued incomplete gamma function $\Gamma(x,5)$ with $x \in [3,4]$ (both of which are readily computable to a high degree of accuracy), only involve slowly-varying elementary functions, and do not involve large/infinite sums or integrals. The only minor complications are mild removable singularities at the boundaries of the regions under consideration, and these are easily handled numerically by a limiting process.

\section{Integral lower bounds}
\label{sec:int-lower-bnds}
From  Euler's reflection formula
\begin{equation}
\label{eq08}
\frac{1}{\Gamma(z)\Gamma(1-z)}=\frac{\sin(\pi z)}{\pi},
\end{equation}
\cref{eq02} and \cite[Eqs. 5.2.5, 5.2.6, 15.8.6]{NIST:DLMF} we obtain for positive integer $n$
\begin{equation}
\label{eq09}
A_n(\alpha,\beta,\omega)
=\frac{(-1)^{n+1}}{\pi}
\omega^n w_{n}(\alpha,\beta,-1/\omega),
\end{equation}
where
\begin{equation}
\label{eq12}
w_{n}(\alpha,\beta,x)
=\frac{1}{n!} \Gamma(n-\alpha)\Gamma(1+\alpha)
\sin(\pi\alpha)\,
{}_2F_1\!\left(\beta,-n;\alpha-n+1;x\right).
\end{equation}

Letting $x=-1/\omega=\exp(i\phi)$, so that
\begin{equation}
\label{eq10}
\phi=\pi-\theta
\quad (\mathrm{mod} \; 2\pi),
\end{equation}
and then appealing to the Schwarz reflection principle, we see that Brannan's conjecture is equivalent to
\begin{equation}
\label{eq11}
w_{n}(\alpha,\beta,-1) \ge  \left|w_{n}(\alpha,\beta,e^{i\phi})\right|,
\end{equation}
for $\alpha,\beta\in(0,1]$, $\phi \in [0,\pi]$ and $n=3,5,7,\ldots$.

Note from \cref{eq01,eq04,eq09,eq12} that for $n=1,2,3,\ldots$
\begin{equation}
\label{eq13}
\lim_{\alpha\to 0} w_{n}(\alpha,\beta,x)
=-\frac{\pi\,\Gamma(\beta+n)}{\Gamma(\beta)\,n!}\,x^{n}.
\end{equation}
and for $n=2,3,4\ldots$
\begin{equation}
\label{eq14}
\lim_{\alpha\to 1} w_{n}(\alpha,\beta,x)
= \frac{\pi\,\Gamma(\beta+n-1)}{\Gamma(\beta)\,n!}
\left\{ n-(\beta+n-1)x \right\}\,x^{n-1}.
\end{equation}

We shall prove the following.
\begin{theorem}
\label{thm:main}
The inequality \cref{eq11} holds for $\alpha, \beta \in (0,1]$, $\phi \in [\phi_0 ,\pi]$ and $n=5,7,9,\ldots$, where $\phi_0=0.061$.
\end{theorem}

We shall attack this  problem using Watson's lemma for large $n$. The following representation will enable us to do so. In this, and throughout, all complex powers use the principal branches.

\begin{lemma}
\label{lem:wnLaplace}
Assume $\alpha,\beta \in [0,1]$, $1 \leq n \in \mathbb{N}$, and $x\in\mathbb{C}\setminus[0,\infty)$. Then
\begin{multline}
\label{eq15}
w_{n}(\alpha,\beta,x)=(-1)^{n+1}\sin(\pi\beta)
(-x)^{\,n-\alpha}\int_{0}^{\infty} 
(e^{s}-1)^{-\beta}\,(e^{s}-x)^{\alpha}e^{-ns}\,ds
\\
+ \sin(\pi\alpha)(-x)^{-\beta}
\int_{0}^{\infty} (e^{s}-1)^{\alpha}
\left(e^{s}-1/x\right)^{-\beta}e^{-ns} ds.
\end{multline}
When $\beta=1$ it is understood that the limit of the first term on the RHS applies.
\end{lemma}

\begin{proof}
From \cref{eq01} and Cauchy's integral formula, for any $c\in(0,1)$ we have for $|\omega| <1$
\begin{equation}
\label{eq16a}
A_n(\alpha,\beta,\omega)
=\frac{1}{2\pi i}\oint_{|z|=c}\frac{(1+\omega z)^{\alpha}}{(1-z)^{\beta}}\,
\frac{dz}{z^{n+1}},
\end{equation}
with the contour positively orientated. Setting $\omega=-\xi$ and using \cref{eq09} (with $x=1/\xi=-1/\omega$) yields
\begin{equation}
\label{eq16}
w_{n}(\alpha,\beta,1/\xi)
=-\frac{1}{2i\xi^{n}}\oint_{|z|=c}\frac{(1-\xi z)^{\alpha}}{(1-z)^{\beta}}\,
\frac{dz}{z^{n+1}}.
\end{equation}
Now we temporarily assume $\xi=e^{i\varphi}$ with $\varphi\in(0,\pi]$
and $\alpha,\beta\in(0,1)$, and make the change of variable
$z=t/\xi$. Then $dz=dt/\xi$ and $z^{-n-1}=\xi^{n+1}t^{-n-1}$, so
\cref{eq16} becomes
\begin{equation}
w_n(\alpha,\beta,1/\xi)
=-\frac{1}{2i}\oint_{|t|=c}
t^{-n-1}(1-t)^\alpha
\left(1-\frac{t}{\xi}\right)^{-\beta}\,dt .
\end{equation}
Since all powers are taken on the principal branches, we have
\begin{equation}
\left(1-\frac{t}{\xi}\right)^{-\beta}
=\left(-\frac1{\xi}\right)^{-\beta}(t-\xi)^{-\beta}
=(-\xi)^\beta(t-\xi)^{-\beta}.
\end{equation}
Thus
\begin{equation}
\label{eq17}
w_n(\alpha,\beta,1/\xi)
=-\frac{1}{2i}(-\xi)^\beta
\oint_{|t|=c}
t^{-n-1}(1-t)^\alpha(t-\xi)^{-\beta}\,dt
=\frac12 i\,(-\xi)^\beta I(\xi),
\end{equation}
where
\begin{equation}
\label{eq18}
I(\xi)=\oint_{|t|=c} 
\frac{t^{-n-1}(1-t)^{\alpha}}{(t-\xi)^{\beta}}\,dt,
\end{equation}
with the contour again positively orientated and $c \in (0,1)$.

Next, $\xi=e^{i\varphi}$ lies off the cut $[1,\infty)$, and the branch points $t=1$ and $t=\xi$ are distinct.  
The contour initially lies inside the unit circle $|t|=1$ and so avoids the branch point $t=\xi$. We take the branch cut for $(1-t)^{\alpha}$ to be $[1,\infty)$ (principal branch on $\mathbb{C}\setminus[1,\infty)$), and the branch cut for $(t-\xi)^{-\beta}$ to be the ray extending from $\xi$ to $\infty$ with $\arg(t)=\arg(\xi)=\varphi$.

Since we assume that $\alpha,\beta \in (0,1)$, the branch points at $t=1$ and $t=\xi$ are locally integrable, and as $t\to\infty$ the integrand is $\mathcal{O}(t^{-n-1+\alpha-\beta})$, so contributions from circular arcs vanish at infinity.
Hence, the loop in \cref{eq18} may be deformed to four rays, two on either side of both cuts. Consequently, we arrive at 
\begin{equation}
\label{eq19}
I(\xi)=I_{\xi}^{+}(\xi)-I_{\xi}^{-}(\xi)
+I_{1}^{+}(\xi)-I_{1}^{-}(\xi),
\end{equation}
where $I_{\xi}^{\pm}(\xi)$ denotes the integral of $t^{-n-1}(1-t)^{\alpha}(t-\xi)^{-\beta}$ from $t=\xi$ to $t=\infty \exp(i\varphi)$ with  $\arg(\xi-t)=\arg (\xi) \mp \pi$, and $I_{1}^{\pm}(\xi)$ denotes the integral of the same integrand from $t=1$ to $t=\infty$ with $\arg(1-t)=\mp \pi$.

For the first pair make the change of variable $t=\xi v$ with $1 \le v < \infty$, so that for $I_{\xi}^{\pm}(\xi)$
\begin{equation*}
(t-\xi)^{-\beta}
=(-\xi)^{-\beta}(1-v)^{-\beta}
=e^{\pm i\pi\beta}(-\xi)^{-\beta}(v-1)^{-\beta}
\quad (1 \le v < \infty).
\end{equation*}
Therefore, since $(1-t)^{\alpha}=(1-\xi v)^{\alpha}$ takes the same value for both integrals for $\varphi\in(0,\pi)$, we obtain
\begin{equation}
\label{eq21}
I_{\xi}^{+}(\xi)-I_{\xi}^{-}(\xi)
=2i (-1)^n\sin(\pi\beta)(-\xi)^{-n-\beta}
\int_{1}^{\infty} v^{-n-1}(v-1)^{-\beta}(1-\xi v)^{\alpha}\,dv.
\end{equation}
Now for $v$ above  and below the cut $[1,\infty)$
\begin{equation*}
(1-v)^{\alpha}=e^{\mp i\pi\alpha}(v-1)^{\alpha},
\end{equation*}
and therefore (with $t$ replaced by $v$ for consistency)
\begin{equation}
\label{eq23}
I_{1}^{+}(\xi)-I_{1}^{-}(\xi)
=-2i\sin(\pi\alpha)\int_{1}^{\infty}
v^{-n-1}(v-1)^{\alpha}
(v-\xi)^{-\beta}dv.
\end{equation}
Inserting \cref{eq21,eq23} into \cref{eq19} yields
\begin{multline}
\label{eq24}
I(\xi)=2i(-1)^n\sin(\pi\beta)(-\xi)^{-n-\beta}
\int_{1}^{\infty} v^{-n-1}(v-1)^{-\beta}
(1-\xi v)^{\alpha}\,dv
\\
-2i\sin(\pi\alpha)\int_{1}^{\infty}
v^{-n-1}(v-1)^{\alpha}
(v-\xi)^{-\beta}dv.
\end{multline}

Now plug \cref{eq24} into \cref{eq17} and in each of the
resulting integrals set $v=e^{s}$, so that $dv=e^{s}\,ds$ and
$v^{-n-1}dv=e^{-ns}\,ds$. Then, on using $(1-\xi e^{s})^{\alpha}=(-\xi)^{\alpha}(e^{s}-\xi^{-1})^{\alpha}$ and replacing $\xi$ by $1/x$, we confirm that \cref{eq15} holds for $\alpha,\beta \in (0,1)$ and $\arg(1/x)=\varphi \in (0,\pi]$.
Finally, the identity extends to $\alpha,\beta\in[0,1]$ by continuity (all boundary singularities are removable, with the case $\beta=1$ interpreted by the limiting value of the first term), and it extends from $x=e^{-i\varphi}$, $\varphi\in(0,\pi]$, to all $x\in\mathbb{C}\setminus[0,\infty)$ by analytic continuation in $x$.
\end{proof}

Our next step is to use this representation to give a suitable lower bound for the difference of the two terms in \cref{eq11}.

\begin{lemma}\label{lem:h-lower-bound}
Let $\alpha,\beta \in (0,1]$, $1 \le n\in\mathbb{N}$ be odd, $\phi \in [0,\pi)$, and define
\begin{equation}
\label{eq25}
\Delta = \pi - \phi > 0,
\end{equation}
\begin{equation}
\label{eq26}
R(s,\phi)
= \left|e^s-e^{\pm i\phi}\right|
= \sqrt{e^{2s} + 1 - 2e^s\cos (\phi)},
\end{equation}
\begin{equation}
\label{eq27}
I_1(\alpha,\beta,\phi;n)
=\frac{1}{\Delta^{2}}
\int_0^\infty \left(e^{s}-1\right)^{-\beta}
\left\{\left(e^{s}+1\right)^{\alpha}
-R(s,\phi)^{\alpha}\right\}e^{-ns}ds,
\end{equation}
\begin{equation}
\label{eq28}
I_2(\alpha,\beta,\phi;n)
= \frac{1}{\Delta^{2}}
\int_0^\infty \left(e^{s}-1\right)^{\alpha}
\left\{\left(e^{s}+1\right)^{-\beta}
-R(s,\phi)^{-\beta}\right\}e^{-ns}ds,
\end{equation}
and
\begin{equation}
\label{eq29}
h(\alpha,\beta,\phi;n)
= \sin(\pi\beta)\,I_1(\alpha,\beta,\phi;n)
+ \sin(\pi\alpha)\,I_2(\alpha,\beta,\phi;n).
\end{equation}
Then
\begin{equation}
\label{eq30}
\Delta^{-2}
\left\{w_{n}(\alpha,\beta,-1)
-\left| w_{n}(\alpha,\beta,e^{i\phi}) 
\right|\right\}
\ge  h(\alpha,\beta,\phi;n).
\end{equation}

\end{lemma}

\begin{remark}
We introduced the factor $\Delta^{-2}$ for convenience, in order that both sides of the inequality \cref{eq30} do not vanish as $\phi \to \pi$, while remaining bounded in this limit, as we show shortly.
\end{remark}

\begin{proof}
In \cref{eq15} set $x=-1$ and $x=\exp(i \phi)$, taking absolute values and then the integral triangle inequality on the latter. On subtraction of the two we deduce that
\begin{multline}
\label{eq31}
w_{n}(\alpha,\beta,-1)
-\left| w_{n}(\alpha,\beta,e^{i\phi}) \right|
\\
\ge
\sin(\pi\beta)\,
\int_{0}^{\infty} 
(e^{s}-1)^{-\beta}\,\left\{ (e^{s}+1)^{\alpha}
-\left|e^{s}-e^{i\phi}\right|^{\alpha} 
\right\}e^{-ns}\,ds \\
\quad
+ \sin(\pi\alpha)\,
\int_{0}^{\infty} 
(e^{s}-1)^{\alpha}\,
\left\{ (e^{s}+1)^{-\beta}-\left| e^{s}
-e^{-i\phi} \right|^{-\beta} \right\}e^{-ns}\,ds,
\end{multline}
and \cref{eq30} then follows on division of both sides by $\Delta^2$.
\end{proof}

\section{Two-term Watson-type expansions for \texorpdfstring{$I_1(\alpha,\beta,\phi;n)$}{} and \texorpdfstring{$I_2(\alpha,\beta,\phi;n)$}{}}
\label{sec:Watson}

Our Watson-type approximation with explicit remainder terms reads as follows.
\begin{lemma}
\label{lem:I1I2-expansion}
Let
\begin{equation}
\label{eq32}
C_1(\alpha,\phi) = \Delta^{-2}
\left[2^\alpha - \left\{2\sin(\tfrac12 \phi)
\right\}^\alpha\right],
\end{equation}
\begin{equation}
\label{eq33}
C_2(\beta,\phi) 
= \Delta^{-2}
\left[\left\{2\sin(\tfrac12 \phi)
\right\}^{-\beta}
- 2^{-\beta}\right],
\end{equation}
\begin{equation}
\label{eq34}
K_1(\alpha,\beta,\phi;s)
= \frac{1}{s^{2}}\left[\left(\frac{s}{e^s-1}\right)^{\beta}
\frac{(e^s+1)^\alpha - R(s,\phi)^\alpha}
{2^\alpha - \left\{2\sin(\tfrac12 \phi)\right\}^\alpha}- 1
- \frac{\alpha-\beta}{2}\,s
\right],
\end{equation}
and
\begin{equation}
\label{eq35}
K_2(\alpha,\beta,\phi;s)
= \frac{1}{s^{2}}\left[\left(\frac{e^s-1}{s}\right)^{\alpha}
\frac{R(s,\phi)^{-\beta}-(e^s+1)^{-\beta}}
{\left\{2\sin(\tfrac12 \phi)\right\}^{-\beta}
-2^{-\beta} }- 1
- \frac{\alpha-\beta}{2}\,s
\right],
\end{equation}
with $R(s,\phi)$ defined by \cref{eq26}. Then, for $\alpha,\beta \in (0,1)$ and $\phi \in [\phi_{0},\pi)$ for fixed $\phi_{0}>0$,
\begin{equation}
\label{eq36}
I_1(\alpha,\beta,\phi;n)
= \frac{C_1(\alpha,\phi)\,\Gamma(1-\beta)}{n^{1-\beta}}
\left\{1+\frac{(\alpha-\beta)(1-\beta)}{2n}\right\}
+ E_{1}(\alpha,\beta,\phi;n),
\end{equation}
and
\begin{equation}
\label{eq37}
I_2(\alpha,\beta,\phi;n)
= -\frac{C_2(\beta,\phi)\,\Gamma(1+\alpha)}{n^{1+\alpha}}
\left\{1+\frac{(\alpha-\beta)(1+\alpha)}{2n}\right\}
+ E_{2}(\alpha,\beta,\phi;n),
\end{equation}
where
\begin{equation}
\label{eq38}
E_{1}(\alpha,\beta,\phi;n)
= C_1(\alpha,\phi)\int_0^\infty s^{2-\beta}  K_1(\alpha,\beta,\phi;s)e^{-ns}\,ds,
\end{equation}
and
\begin{equation}
\label{eq39}
E_{2}(\alpha,\beta,\phi;n)
= -C_2(\beta,\phi)\int_0^\infty s^{2+\alpha}  K_2(\alpha,\beta,\phi;s)e^{-ns}\,ds.
\end{equation}
\end{lemma}

\begin{remark}
We have
\begin{equation}
\label{eq40}
K_1(\alpha,\beta,\phi;s),\ K_2(\alpha,\beta,\phi;s) = \mathcal{O}(1)
\quad\text{as }s\to0,
\end{equation}
uniformly for $\alpha,\beta \in (0,1)$ and $\phi \in [\phi_{0},\pi)$. In particular, all small-$s$
asymptotics are uniform in $\phi$ up to $\phi\to\pi$. The following limits as $\phi \to \pi$ and $\alpha+\beta \to 0$ also apply:
\begin{equation}
\label{eq41}
K_1(\alpha,\beta,\pi;s)
=\frac{1}{s^{2}}\Biggl[
\left(\frac{s}{e^{s}-1}\right)^{\beta}
\frac{4e^{s}}{(e^{s}+1)^{2}}
\left(\frac{e^{s}+1}{2}\right)^{\alpha}
-1-\frac{\alpha-\beta}{2}\,s
\Biggr],
\end{equation}
\begin{equation}
\label{eq42}
K_2(\alpha,\beta,\pi;s)
=\frac{1}{s^{2}}\Biggl[
\left(\frac{e^{s}-1}{s}\right)^{\alpha}
\frac{4e^{s}}{(e^{s}+1)^{2}}
\left(\frac{2}{e^{s}+1}\right)^{\beta}
-1-\frac{\alpha-\beta}{2}\,s
\Biggr],
\end{equation}
and
\begin{equation}
\label{eq43}
K_1(0,0,\phi;s)
=K_2(0,0,\phi;s)
=\frac{1}{s^{2}}\left[ 
\frac{\ln\left\{\left(e^{s}+1\right)/R(s,\phi)
\right\}}
{\ln\left\{\csc(\tfrac12 \phi)\right\}}
-1 \right]
\quad (0<\phi<\pi).
\end{equation}
Furthermore, 
\begin{equation}
\label{eq44}
K_1(\alpha,\beta,\phi;s)
=-\frac{\alpha-\beta}{2s}-\frac{1}{s^{2}}
+\mathcal{O}\left(s^{\beta-2}e^{-(1-\alpha+\beta)s}\right),
\end{equation}
and
\begin{equation}
\label{eq45}
K_2(\alpha,\beta,\phi;s)
=-\frac{\alpha-\beta}{2s}-\frac{1}{s^{2}}
+\mathcal{O}\left(s^{-\alpha-2}e^{-(1-\alpha+\beta)s}\right),
\end{equation}
as $s\to\infty$.

The coefficients $C_1(\alpha,\phi)$ and $C_2(\beta,\phi)$ defined in \cref{eq32,eq33} remain smooth and strictly positive for $\phi\in[\phi_{0},\pi)$ and $\alpha,\beta \in (0,1)$,
and have finite nonnegative limits as $\phi\to\pi$; in particular
\begin{equation}
\label{eq46}
C_1(\alpha,\phi)=\frac{2^\alpha\,\alpha}{8}
+ \mathcal{O}\left\{(\phi - \pi)^{2}\right\},
\end{equation}
and
\begin{equation}
\label{eq47}
C_2(\beta,\phi)=\frac{2^{-\beta}\,\beta}{8}
+ \mathcal{O}\left\{(\phi - \pi)^{2}\right\},
\end{equation}
as $\phi \to \pi$.
\end{remark}

\begin{proof}
From \cref{eq27,eq28}
\begin{equation}
\label{eq48}
I_j(\alpha,\beta,\phi;n)
=\int_0^\infty 
G_j(\alpha,\beta,\phi;s) 
e^{-ns}ds
\quad (j=1,2),
\end{equation}
where
\begin{equation}
\label{eq49}
G_1(\alpha,\beta,\phi;s)
= \frac{(e^{s}+1)^\alpha 
- R(s,\phi)^\alpha}{\Delta^2(e^s-1)^\beta},
\end{equation}
and
\begin{equation}
\label{eq50}
G_2(\alpha,\beta,\phi;s)
= \frac{(e^s-1)^{\alpha}}{\Delta^2}
\left[(e^{s}+1)^{-\beta}-R(s,\phi)^{-\beta}\right].
\end{equation}
Using a Taylor expansion at $s=0$ and the definition \cref{eq32,eq33} one finds that, as $s \to 0$,
\begin{equation*}
s^{\beta}G_1,\,
s^{-\alpha}G_2
= \pm C_{1,2}\left\{ 
1+ \tfrac12 (\alpha-\beta)s
+ \mathcal{O}\left(s^{2}\right) \right\},
\end{equation*}
uniformly for $\alpha,\beta \in (0,1)$ and $\phi\in[\phi_{0},\pi)$. We therefore can write the integrands of \cref{eq48} in the forms
\begin{equation}
\label{eq52}
G_1(\alpha,\beta,\phi;s)
= C_{1}(\alpha,\phi)s^{-\beta}\left\{1
+\tfrac12(\alpha-\beta)s
+s^{2}K_1(\alpha,\beta,\phi;s)\right\},
\end{equation}
and
\begin{equation}
\label{eq53}
G_2(\alpha,\beta,\phi;s)
=- C_{2}(\beta,\phi)s^{\alpha}\left\{1
+\tfrac12(\alpha-\beta)s
+s^{2}K_2(\alpha,\beta,\phi;s)\right\},
\end{equation}
where the remainder kernels $K_{1,2}(\alpha,\beta,\phi;s)$ are given by \cref{eq34,eq35} and satisfy \cref{eq40}. Substituting \cref{eq52,eq53}  into \cref{eq27,eq28}, and using
\begin{equation}
\label{eq54}
\int_0^\infty s^{\mu} e^{-n s}\,ds
= \frac{\Gamma(1+\mu)}{n^{1+\mu}}
\quad (\Re(n)>0, \Re(\mu)>-1),
\end{equation}
we obtain \cref{eq36,eq37}.
\end{proof}

Let us now consider the RHS of \cref{eq30}, suitably scaled. We introduce the prefactors
\begin{equation}
\label{eq55}
A_1(\alpha,\phi) =\frac{ (1-\alpha)C_1(\alpha,\phi)}
{\sin(\pi\alpha)\Gamma(1+\alpha)}
=\frac{(1-\alpha)\left[ 2^\alpha 
- \left\{2\sin(\tfrac12 \phi)\right\}^\alpha \right]}
{\sin(\pi\alpha)\Gamma(1+\alpha)(\pi-\phi)^2},
\end{equation}
and
\begin{equation}
\label{eq56}
A_2(\alpha,\beta,\phi) 
= \frac{ (1-\alpha)C_2(\beta,\phi)}
{\sin(\pi\beta)\Gamma(1-\beta)}
=\frac{(1-\alpha)
\Gamma(\beta)
\left[\left\{2\sin(\tfrac12 \phi)
\right\}^{-\beta}- 2^{-\beta}\right]}
{\pi(\pi-\phi)^2}.
\end{equation}
Then, on combining \cref{eq29,eq36,eq37,eq38,eq39,eq55,eq56}, one arrives at:

\begin{theorem}
For $\alpha,\beta \in [0,1]$ and $\phi \in (0,\pi]$    
\begin{equation}
\label{eq57}
\mathcal{H}(\alpha,\beta,\phi;n)
= \mathcal{H}^{(2)}(\alpha,\beta,\phi;n)
+\mathcal{E}_{0}(\alpha,\beta,\phi;n)
+\mathcal{E}_{\infty}(\alpha,\beta,\phi;n),
\end{equation}
where
\begin{equation}
\label{eq58}
\mathcal{H}(\alpha,\beta,\phi;n)
= \frac{(1-\alpha)n^{1+\alpha-\beta}h(\alpha,\beta,\phi;n)}
{(\alpha+\beta)\sin(\pi\alpha)
\sin(\pi\beta)\Gamma(1-\beta)\Gamma(1+\alpha)},
\end{equation}
\begin{multline}
\label{eq59}
\mathcal{H}^{(2)}(\alpha,\beta,\phi;n)
= \frac{1}{\alpha+\beta}\left[
A_1(\alpha,\phi) n^\alpha
\left\{1+\frac{(\alpha-\beta)(1-\beta)}{2n}\right\}
\right.
\\
\left.
- \frac{A_2(\alpha,\beta,\phi)}{n^{\beta}} 
\left\{1+\frac{(\alpha-\beta)(1+\alpha)}{2n}\right\}\right],
\end{multline}
\begin{equation}
\label{eq60}
\mathcal{E}_{0}(\alpha,\beta,\phi;n)
= n^{1+\alpha-\beta} \int_0^1 
L(\alpha,\beta,\phi;s)e^{-ns}ds,
\end{equation}
and
\begin{equation}
\label{eq61}
\mathcal{E}_{\infty}(\alpha,\beta,\phi;n)
= n^{1+\alpha-\beta} \int_1^\infty L(\alpha,\beta,\phi;s)
e^{-ns}ds,
\end{equation}
in which
\begin{equation}
\label{eq62}
L(\alpha,\beta,\phi;s)
= \frac{s^2}{\alpha+\beta}
\left\{\frac{A_1(\alpha,\phi)
K_1(\alpha,\beta,\phi;s)}
{\Gamma(1-\beta)s^{\beta}}
- \frac{s^{\alpha}A_2(\alpha,\beta,\phi)
K_2(\alpha,\beta,\phi;s)}
{\Gamma(1+\alpha)}\right\},
\end{equation}
and in all cases (finite) limits applying when $\alpha \to 0,1$, $\beta \to 0,1$ and/or $\phi \to \pi$. In \cref{eq62} $K_1(\alpha,\beta,\phi;s)$ and $K_2(\alpha,\beta,\phi;s)$ are given by \cref{eq26,eq34,eq35}.
\end{theorem}

Note, from \cref{eq55,eq56},
\begin{equation}
\label{eq63}
A_1(0,\phi)
=\frac{\ln\left\{\csc(\tfrac12 \phi)\right\}}
{\pi(\pi-\phi)^2},
\end{equation}
\begin{equation}
\label{eq64}
A_1(1,\phi) =\frac{2
\left\{1 - \sin(\tfrac12 \phi)\right\}}
{\pi\Delta^2},
\end{equation}
\begin{equation}
\label{eq65}
A_1(\alpha,\pi)
=\frac{2^{\alpha}\Gamma(2-\alpha)}{8\pi},
\end{equation}
\begin{equation}
\label{eq66}
A_2(\alpha,0,\phi)
=\frac{(1-\alpha)\ln\left\{\csc(\tfrac12 \phi)\right\}}
{\pi(\pi-\phi)^2},
\end{equation}
\begin{equation}
\label{eq67}
A_2(\alpha,1,\phi)
=\frac{ (1-\alpha)
\left\{\csc(\tfrac12 \phi)
- 1\right\}}{2\pi \Delta^2},
\end{equation}
and
\begin{equation}
\label{eq68}
A_2(\alpha,\beta,\pi)
=\frac{(1-\alpha)\Gamma(1+\beta)}{8\pi\,2^{\beta}},
\end{equation}
and $L(\alpha,\beta,\pi;s)$ follows from \cref{eq41,eq42,eq62,eq65,eq68}.

Furthermore, using the $\mathcal{O}(1)$ behaviour of $K_1(\alpha,\beta,\phi;s)$ and $K_2(\alpha,\beta,\phi;s)$ as $s\to0$, one finds that
\begin{equation}
\label{eq69}
L(\alpha,\beta,\phi;s)
= \mathcal{O}\left(s^{2-\beta}\right)
\quad  (s \to 0),
\end{equation}
for $\alpha,\beta  \in (0,1]$ and $\phi \in [\phi_{0},\pi]$. Note from \cref{eq56} and $1/\Gamma(0)=0$ that  $L(1,1,\phi;s)=0$. Moreover the limit $L(0,0,\phi;s)$ exists and is finite for all $\phi\in[\phi_{0},\pi)$ and bounded
$s \ge 0$. In particular, we have
\begin{multline}
\label{eq70}
L(0,0,\phi;s)
=\frac{1}{2\pi(\pi-\phi)^2}
\Biggl[2\ln\left\{\frac{R(s,\phi)}{e^{s}+1}\right\}
\ln\left(e^{s}-1\right)
+\ln^{2}\left(e^{s}+1\right)
\biggr. \\ \left.
-\ln^{2}\left\{R(s,\phi)\right\}
+\ln\left\{\sin(\tfrac{1}{2}\phi)\right\}
\ln\left\{\frac{4\sin(\tfrac{1}{2}\phi)}
{s^{2}}\right\}
\right]
\\
=\frac{\left\{1+\cos(\phi)\right\}s^{2}\ln(s)}
{8\pi\left\{1-\cos(\phi)\right\}(\pi-\phi)^2}
+\mathcal{O}\left(s^{2}\right)
\quad (s \to 0),
\end{multline}
where $R(s,\phi)$ is given by \cref{eq26}. From this
\begin{multline}
\label{eq71}
L(0,0,\pi;s)
=\frac{1}{8\pi}
\left[\mathrm{sech}^{2}(\tfrac{1}{2}s)
\ln\left\{\coth(\tfrac{1}{2}s)\right\}
+\ln(\tfrac{1}{2}s)
\right]
\\
=\frac{s^{2} \ln(s)}{32\pi}
+\mathcal{O}\left(s^{2}\right)
\quad (s \to 0).
\end{multline}

We also note from \cref{eq25,eq26,eq34,eq35,eq62,eq64,eq67} that
\begin{equation}
\label{eq172}
L(0,1,\phi;s)
=\frac{1}{\pi(\pi-\phi)^2}\left\{
\frac{\csc(\tfrac12 \phi)-1}{2}
\left(1-\frac{s}{2}\right)
-\frac{1}{R(s,\phi)}
+\frac{1}{e^{s}+1}
\right\},
\end{equation}
\begin{equation}
\label{eq173}
L(1,0,\phi;s)
=\frac{1}{\pi(\pi-\phi)^2}
\left\{\sin(\tfrac12 \phi)(s+2)
- R(s,\phi)+e^{s}-s-1
\right\},
\end{equation}
and from \cref{eq44,eq45,eq62} 
\begin{equation}
\label{eq72}
L(\alpha,\beta,\phi;s)
\sim
\frac{\alpha-\beta}{2(\alpha+\beta)}\left\{
\frac{A_2(\alpha,\beta,\phi)}{\Gamma(1+\alpha)}\,s^{1+\alpha}
-\frac{A_1(\alpha,\phi)}{\Gamma(1-\beta)}\,s^{1-\beta}
\right\}
\quad (s\to\infty).
\end{equation}

Our task now is to bound, from below, the remainder terms $\mathcal{E}_{0}(\alpha,\beta,\phi;n)$ and $\mathcal{E}_{\infty}(\alpha,\beta,\phi;n)$. The following constants play an important role in this.

\begin{proposition}
\label{prop:m0minf}
For $\phi_{0}=0.061$ let
\begin{equation}
\label{eq73}
m_{0}
= \inf_{\substack{\alpha,\beta \in (0,1)\\ 
\phi \in [\phi_{0},\pi),\ s \in [0,1]}}
 L_0(\alpha,\beta,\phi;s),
\end{equation}
where
\begin{equation}
\label{eq74}
 L_0(\alpha,\beta,\phi;s)
=s^{-1+\beta}\,L
(\alpha,\beta,\phi;s),
\end{equation}
with $ L_0(\alpha,\beta,\phi;0)=\lim_{s \to 0}  L_0(\alpha,\beta,\phi;s)=0$, and
\begin{equation}
\label{eq75}
m_{\infty}
= \inf_{\substack{\alpha,\beta \in (0,1)\\ 
\phi \in [\phi_{0},\pi),\ s \in [1,\infty)}}
 L_\infty(\alpha,\beta,\phi;s),
\end{equation}
where
\begin{equation}
\label{eq76}
 L_\infty(\alpha,\beta,\phi;s)
=s^{-2-\alpha}\,L(\alpha,\beta,\phi;s).
\end{equation}
Then
\begin{equation}
\label{eq77}
m_0 = -0.0976382271\cdots,
\end{equation}
attained at $\alpha=0$, $\beta=0$, $\phi=\phi_{0}$,
$s=s_0:=0.0215923350\cdots$ (cf. \cref{eq70}), and
\begin{equation}
\label{eq78}
m_{\infty} = -0.03332478558 \cdots,
\end{equation}
attained at $\alpha=0$, $\beta=1$, $\phi=\phi_{0}$,
$s=s_\infty:=3.999154909\cdots$ (cf. \cref{eq172}). 
\end{proposition}

\begin{remark}
On the compact set $\alpha,\beta \in [0,1]$, $\phi \in [\phi_{0},\pi]$ and $s \in [0,1]$ the function $ L_0(\alpha,\beta,\phi;s)$ is continuous, bounded and $\mathcal{O}(s)$ as $s \to 0$. In \cref{eq74} we introduced the factor $s^{-1+\beta}$ rather than $s^{-2+\beta}$, which guarantees that $ L_0(\alpha,\beta,\phi;s)$ vanishes at $s=0$, and subsequently yields the numerically small (in absolute value) lower bound $m_0$ as shown. The price is that in the error bounds of \cref{lem:error-bounds} below there is a factor $n^{-1+\alpha}$ rather than a smaller term $n^{-2+\alpha}$, but this weaker decay is more than sufficient for $n \ge 5$ which we consider. 

Likewise, we use the factor $s^{-2-\alpha}$ as opposed $s^{-1-\alpha}$ in the definition \cref{eq76} (cf. \cref{eq72}) for the same reason. On the set set $\alpha,\beta \in [0,1]$, $\phi \in [\phi_{0},\pi]$ and $s \in [1,\infty)$ the function $ L_\infty(\alpha,\beta,\phi;s)$ is continuous, bounded and $\mathcal{O}(s^{-1})$ as $s \to \infty$. 

See \cref{fig:m0,fig:minf} for graphs of $L_0(\alpha,\beta,\phi_0;s_0)$ and $L_\infty(\alpha,\beta,\phi_0;s_\infty)$ for $\alpha,\beta \in [0,1]$. In \cite{Dunster:2026:BWM} we provide plots of $L_0(\alpha,\beta,\phi_0;s)$ and $L_0(\alpha,\beta,\pi;s)$ animated for $s \in [0,1]$, and $L_\infty(\alpha,\beta,\phi_0;s)$ and $L_\infty(\alpha,\beta,\pi;s)$ animated for $s \in [1,30]$.
\end{remark}

\begin{figure}
 \centering
 \includegraphics[
 width=0.9\textwidth,keepaspectratio]{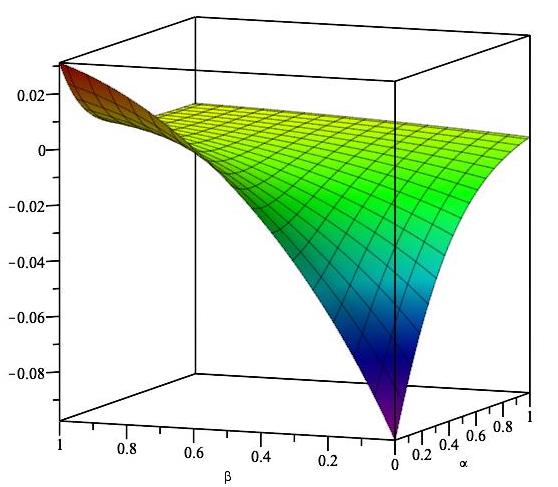}
 \caption{Graph of $ L_0(\alpha,\beta,\phi_0;s_0)$ for $\alpha,\beta \in [0,1]$.}
 \label{fig:m0}
\end{figure}

\begin{figure}
 \centering
 \includegraphics[
 width=0.9\textwidth,keepaspectratio]{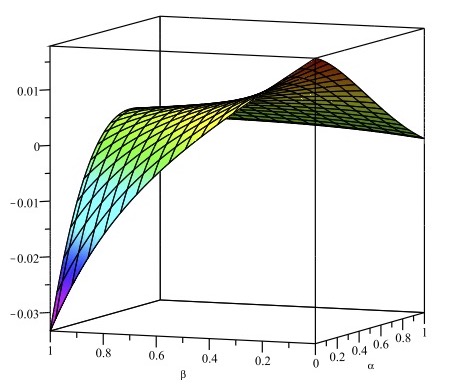}
 \caption{Graph of $ L_\infty(\alpha,\beta,\phi_0;s_\infty)$ for $\alpha,\beta \in [0,1]$.}
 \label{fig:minf}
\end{figure}

Before discussing the computations leading to \cref{eq77,eq78}, we apply these values to obtain the following desired explicit lower bounds.

\begin{lemma}\label{lem:error-bounds}
The error terms in \cref{eq57} satisfy
\begin{equation}
\label{eq81}
\mathcal{E}_{0}(\alpha,\beta,\phi;n)
>
m_0\,\Gamma(2-\beta)\,n^{-1+\alpha},
\end{equation}
and
\begin{equation}
\label{eq82}
\mathcal{E}_{\infty}(\alpha,\beta,\phi;n)
\ge 
m_{\infty} \,\Gamma(3+\alpha,n)\,n^{-2-\beta}.
\end{equation}
uniformly for $n>0$, $\alpha,\beta \in [0,1]$ and $\phi \in [\phi_{0},\pi]$, where $\phi_{0}=0.061$.
\end{lemma}

\begin{proof}
By the definition \cref{eq73} we have under the hypotheses of the lemma 
\begin{equation*}
L(\alpha,\beta,\phi;s)
\;\ge\; m_0\,s^{1-\beta},
\end{equation*}
Inserting this into  \cref{eq60} and using the elementary inequality
\begin{equation}
\label{eq83}
\int_a^b f(s)\,g(s)\,ds
\;\ge\; f_0\int_a^b g(s)\,ds
\quad (f(s) \ge f_0,\, g(s )\ge 0),
\end{equation}
yields, since $m_0<0$,
\begin{equation*}
\mathcal{E}_{0}(\alpha,\beta,\phi;n)
\ge m_0\,n^{1+\alpha-\beta}
\int_0^1 s^{1-\beta}e^{-ns}\,ds
> m_0\,n^{1+\alpha-\beta}
\int_0^\infty s^{1-\beta}e^{-ns}\,ds,
\end{equation*}
and \cref{eq81} then follows from  referring to \cref{eq54}.

Similarly, combining \cref{eq61} with the pointwise lower bound 
\begin{equation*}
L(\alpha,\beta,\phi;s)
\;\ge\; m_{\infty}\,s^{2+\alpha},
\end{equation*}
which comes from \cref{eq75}, and using the same inequality \cref{eq83}, we obtain for $\alpha,\beta \in [0,1]$ and $\phi \in [\phi_{0},\pi]$,
\begin{equation*}
\mathcal{E}_{\infty}(\alpha,\beta,\phi;n)
\ge m_{\infty}\,n^{1+\alpha-\beta}
\int_1^\infty s^{2+\alpha}e^{-ns}\,ds.
\end{equation*}
With the change of variable $s \to s/n$ and referring to the definition of the incomplete gamma function \cite[Eq. 8.2.2]{NIST:DLMF} we arrive at \cref{eq82}.
\end{proof}

\subsection{Numerical computation of \texorpdfstring{$m_0$}{} and \texorpdfstring{$m_\infty$}{}}

We minimized $L_{0}(\alpha,\beta,\phi;s)$ by a direct numerical search in Maple (Digits $=60$) over
$\alpha,\beta\in[0,1]$ and $\phi\in[\phi_{0},\pi]$, with the $s$--dependence handled as follows. We sampled $(\alpha,\beta,\phi)$ on a fine uniform grid (using explicit limiting cases on the boundary faces), with the base grids in $\alpha$ and $\beta$ chosen to be denser near $0$ and $1$, and $\phi$ sampled uniformly on $[\phi_{0},\pi]$. Each cell of the $(\alpha,\beta,\phi)$ grid was refined once by inserting two equally spaced interior points in each coordinate direction.

For each sampled triple $(\alpha,\beta,\phi)$, we regarded $L_{0}(\alpha,\beta,\phi;s)$ as a function of the single variable $s\in[0,1]$, and evaluated it at the endpoint $s=1$ and at any interior critical points $s\in(0,1)$ satisfying $\partial L_{0}(\alpha,\beta,\phi;s)/\partial s=0$. The critical points were located using \texttt{fsolve} on sign-change brackets, based on an exact expression for the $s$-derivative. Since $L_{0}(\alpha,\beta,\phi;0)=0$, the endpoint $s=0$ did not need to be sampled.

The boundary cases $\alpha=0,1$, $\beta=0,1$, and $\phi=\pi$ are removable singularities; these were evaluated stably using the exact limiting formulas in \cref{eq41,eq42,eq43,eq44,eq45,eq62,eq63,eq64,eq65,eq66,eq67,eq68,eq74}.

We carried out a similar grid search for the
normalised kernel $ L_\infty(\alpha,\beta,\phi;s)$ defined in \cref{eq76}, again in Maple with 60-digit arithmetic.
We sampled on the same $\alpha$, $\beta$ and $\phi$ 3-D grid as before, but this time for $s$ at 1 and 30, and at any critical points within $(1,30)$. At each quadruple point we evaluated $ L_\infty(\alpha,\beta,\phi;s)$, storing the current minimum. Over this truncated box, with $s$ taking only the boundary and critical interior values, the minimum recorded value was given by \cref{eq78}.

That we were able to evaluate  $m_{\infty}$ in \cref{eq75} by restricting our search for the infimum of $ L_\infty(\alpha,\beta,\phi;s)$ to $1 \le s \le 30$ is due to the following.
\begin{proposition}
\begin{multline}
\label{eq86}
\inf_{\substack{\alpha,\beta\in(0,1)
\\ \phi \in [\phi_0,\pi),\,s\in[1,30]}} 
\  L_\infty(\alpha,\beta,\phi;s)
=-0.03332478558 \cdots
\\
< \inf_{\substack{\alpha,\beta\in(0,1)
\\ \phi \in [\phi_0,\pi),\,s\in[30,\infty)}} 
 L_\infty(\alpha,\beta,\phi;s)
=-0.008293818653\cdots,
\end{multline}
with the latter infimum attained at at $\alpha=0$, $\beta=1$, $\phi=0.061$, and $s=30$.
\end{proposition}

To examine the large-$s$ regime numerically in a stable way, we re-expressed $ L_\infty(\alpha,\beta,\phi;s)$ in terms of exponentials with negative exponent; for example,
\begin{multline*}
\left(\frac{s}{e^s-1}\right)^{\beta}
\left\{(e^{s}+1)^\alpha - R(s,\phi)^\alpha \right\}
\\
=\left(\frac{s}{1-e^{-s}}\right)^{\beta}
e^{(\alpha-\beta)s}
\left\{(1+e^{-s})^\alpha
-\left(1+e^{-2s}-2e^{-s}\cos(\phi)\right)^{1/2}\right\},
\end{multline*}
which appears in \cref{eq76} via \cref{eq26,eq34,eq62}. We then minimized the resulting large $s$ stable form of $ L_{\infty}(\alpha,\beta,\phi;s)$ over $\alpha,\beta\in(0,1)$, $\phi\in[\phi_{0},\pi)$ and $s\in[30,S]$ using Maple's global search (\texttt{Optimization[Minimize]}) (using high precision arithmetic), for various values of $S \in [100,10^{5}]$. In every run, the optimizer returned the same minimum at the lower boundary $s=30$, with the value given in \cref{eq86}. This provides a practical numerical certification that the infimum defining $m_{\infty}$ is achieved within the finite window $1\le s\le 30$, bearing in mind, from \cref{eq72,eq76}, that $ L_\infty(\alpha,\beta,\phi;s)=\mathcal{O}(s^{-1})$ as $s \to \infty$. This is illustrated by our optimization computation for $s=10^5$ giving the value 
\begin{equation}
\label{eq88}
\inf_{\substack{\alpha,\beta \in (0,1)\\ 
\phi \in [\phi_{0},\pi)}}
 L_\infty(\alpha,\beta,\phi;10^{5})
= L_\infty(0,1,\phi_0;10^{5})
=-0.2665816964\times 10^{-5}.
\end{equation}

Finally, we also used \texttt{Optimization[Minimize]} to independently verify the grid results that yielded \cref{eq77,eq78} for $s \in [0,1]$ and $s \in [1,30]$, respectively. The Maple source code for the evaluation of $m_0$ and $m_\infty$ is provided in \cite{Dunster:2026:BWM}.

\section{Proof of \texorpdfstring{\cref{thm:main}}{} for \texorpdfstring{$n=5$}{}}
\label{sec:P5}
Let $w_{n}(\alpha,\beta,x)$ be the scaled hypergeometric function defined by \cref{eq12}, $A_{1,2}(\alpha,\phi)$ be given by \cref{eq55,eq56}, and $m_0$ and $m_\infty$ be the infima given by \cref{eq77,eq78}. Then from \cref{eq30}, the scaling \cref{eq58} together with Euler's reflection identity \cref{eq08}, and the decomposition \cref{eq57,eq59} combined with the bounds \cref{eq81,eq82}, we deduce the following.

\begin{theorem}
\label{thm:Pn}
For $n=1,3,5,\dots$, $\alpha,\beta \in (0,1]$ and $\phi \in [\phi_{0},\pi]$
\begin{equation}
\label{eq89}
w_{n}(\alpha,\beta,-1)-\left| w_{n}(\alpha,\beta,e^{i\phi}) \right|
\geq Q_n(\alpha,\beta,\phi)\,P_{n}(\alpha,\beta,\phi),
\end{equation}
where
\begin{multline}
\label{eq90}
P_{n}(\alpha,\beta,\phi)=\frac{1}{\alpha+\beta}
\left[ A_1(\alpha,\phi)\,n^\alpha
\left\{1+\frac{(\alpha-\beta)(1-\beta)}{2n}\right\}
\right.
\\
\left.
- \frac{A_2(\alpha,\beta,\phi)}{n^{\beta}}
\left\{ 1+\frac{(\alpha-\beta)(1+\alpha)}{2n} \right\} \right]
\\
-\frac{\left|m_0\right|\,\Gamma(2-\beta)}{n^{1-\alpha}}
-\frac{\left|m_{\infty}\right|\,\Gamma(3+\alpha,n)}{n^{2+\beta}},
\end{multline}
and
\begin{equation}
\label{eq91}
Q_{n}(\alpha,\beta,\phi)
= \frac{\pi\,(\alpha+\beta)
\sin(\pi\alpha)\,\Gamma(1+\alpha)\,(\pi-\phi)^{2}}
{(1-\alpha)\Gamma(\beta)\,n^{1+\alpha-\beta}}
\ge 0.
\end{equation}

\end{theorem}

Note that at the boundaries
\begin{multline}
\label{eq92}
P_n(0,1,\phi)
=\frac{1}{\pi(\pi-\phi)^2}
\left[\ln\left\{\csc\left(\tfrac12 \phi\right)\right\}
-\frac{\csc\left(\tfrac12 \phi\right)-1}
{2 n}
\left(1-\frac{1}{2n}\right)\right]
\\
-\frac{|m_0|}{n}
-\frac{\left|m_{\infty}\right|}{n\,e^{n}}\left(1+\frac{2}{n}
+\frac{2}{n^{2}}\right),
\end{multline}
\begin{equation}
\label{eq93}
P_{n}(1,0,\phi)
= n\left(2+\frac{1}{n}\right)
\frac{1-\sin(\tfrac12 \phi)}
{\pi\,(\pi-\phi)^2}
-\left|m_{0}\right|
-\frac{n\left|m_{\infty}\right|}{e^{n}}\left(1+\frac{3}{n}
+\frac{6}{n^{2}}+\frac{6}{n^{3}}\right),
\end{equation}
\begin{equation}
\label{eq94}
P_{n}(1,1,\phi)= \frac{n
\left\{1-\sin(\tfrac12 \phi)\right\}}{\pi\,(\pi-\phi)^2}
-\left|m_{0}\right|
-\frac{\left|m_{\infty}\right|}{e^{n}}\left(1+\frac{3}{n}
+\frac{6}{n^{2}}+\frac{6}{n^{3}}\right),
\end{equation}
and 
\begin{multline}
\label{eq95}
P_{n}(0,0,\phi)
=\frac{\ln\left\{\csc(\tfrac12\phi)\right\}}
{\pi(\pi-\phi)^2}
\left[\ln(n)+\gamma-\frac{1}{2}
\ln\left\{\frac{\csc(\tfrac12\phi)}{4}\right\}\right]
\\
-\frac{|m_0|}{n}
-\frac{\left|m_{\infty}\right|}{e^{n}}\left(1+\frac{2}{n}
+\frac{2}{n^{2}}\right).
\end{multline}
All the above can easily be shown to be positive for $n \ge 5$ and $\phi \in [\phi_{0},\pi)$, with finite positive limits as $\phi \to \pi$. 

In addition, $P_{n}(\alpha,\beta,\pi)$ is given by \cref{eq90} with $A_1(\alpha,\pi)$ and $A_2(\alpha,\beta,\pi)$ being given by \cref{eq65,eq68}, and incidentally from this (for $n = 5$) we find numerically that
\begin{equation}
\label{eq96}
\inf_{\alpha,\beta \in(0,1)} P_5 (\alpha,\beta,\pi)
=P_5(0,1,\pi)=\frac{9}{80\pi}-\frac{|m_0|}{5}-\frac{37\,|m_\infty|}{125\,e^{5}}
=0.01621575275 \cdots.
\end{equation}

\begin{figure}[H]
 \centering
 \includegraphics[
 width=0.9\textwidth,keepaspectratio]{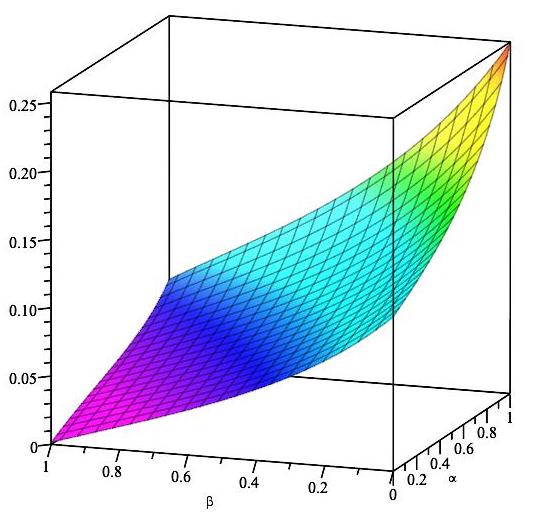}
 \caption{Graph of $P_5(\alpha,\beta,\phi_0)$ for $\alpha,\beta \in [0,1]$.}
 \label{fig:P5}
\end{figure}

The desired positivity of the infimum is given by the following, which then on referring to \cref{thm:Pn}, establishes \cref{thm:main} for $n=5$.

\newpage

\begin{proposition}
\label{prop:Pn-Bound}
For $\phi_0=0.061$
\begin{multline}
\label{eq97}
\inf_{\substack{\alpha,\beta\in(0,1)
\\ \phi \in [\phi_0,\pi)}} 
P_5(\alpha,\beta,\phi)
=P_5(0,1,\phi_0)
\\
=\frac{1}{\pi(\pi-\phi_{0})^2}
\left[\ln\left\{\csc\left(\tfrac12 \phi_{0}\right)\right\}
-\frac{9\left\{ \csc\left(\tfrac12 \phi_{0}\right)-1 \right\}}
{100}\right]
\\
-\frac{|m_0|}{5}
-\frac{37\left|m_{\infty}\right|}{125\,e^{5}}
= 0.001500310752 \cdots.
\end{multline}
\end{proposition}

\begin{figure}
 \centering
 \includegraphics[
 width=0.9\textwidth,keepaspectratio]{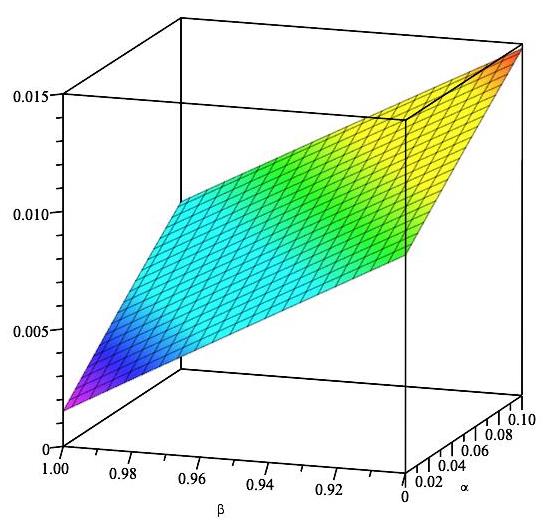}
 \caption{Graph of $P_5(\alpha,\beta,\phi_0)$ for $\alpha \in [0,0.1]$ and $\beta \in [0.9,1]$.}
 \label{fig:P5-Zoom}
\end{figure}

\begin{remark}
See \cref{fig:P5,fig:P5-Zoom} for surface plots of $P_5(\alpha,\beta,\phi_0)$ (the latter depicting the surface close to the minimum). That the (positive) infimum \cref{eq97} is relatively small is a consequence of our choice of $\phi_0=0.061$ being as small as possible. We chose the smallest such value (to three digits) that would still result in \cref{eq97} being positive. We see from \cref{eq56} that $-A_2(\alpha,\beta,\phi) \to -\infty$ as $\phi \to 0^+$. In addition, the infima $m_0$ and $m_\infty$, as given by \cref{eq62,eq73,eq74,eq75,eq76}, also depend on $\phi_0$, and they too approach $-\infty$ as $\phi_0 \to 0^+$. Hence from \cref{eq90} $P_5(\alpha,\beta,\phi) \to -\infty$ in this limit (noting from \cref{eq55} that $A_1(\alpha,\beta,\phi)$ remains bounded).
\end{remark}

To numerically obtain the infimum of $P_5(\alpha,\beta,\phi)$ given in \cref{eq90} we carried out a 2-D $\alpha,\beta$ grid search in Maple (using high-precision arithmetic), storing at each step the smallest current value of $P_5(\alpha,\beta,\phi)$. We fixed $\phi_{0}=0.061$ and $n=5$, and used the numerical values of $m_0$ and $m_{\infty}$ from \cref{eq77,eq78}. The search sampled $\alpha$ and $\beta$ over $60^2$ point pairs on a grid biased toward $0$ and $1$ (including boundary values). 

For each grid pair $(\alpha,\beta)$ we minimized over $\phi\in[\phi_0,\pi]$ by evaluating $P_5(\alpha,\beta,\phi)$ only at the endpoints $\phi=\phi_0$ and $\phi=\pi$, together with any interior stationary points in $(\phi_0,\pi)$ where $\partial P_5(\alpha,\beta,\phi)/ \partial \phi=0$. These stationary points were computed from exact closed-form formulas for the derivatives, using \texttt{fsolve} with sign-change bracketing. Since the cases $\alpha,\beta\in\{0,1\}$ and $\phi=\pi$ are removable singularities, near-edge values were snapped to $\alpha,\beta=0$ or $1$, and $\phi=\pi$, with both $P_5(\alpha,\beta,\phi)$ and $\partial P_5(\alpha,\beta,\phi) /\partial \phi$ being evaluated via the corresponding exact limiting formulas on edges and corners (see \cref{eq65,eq68,eq92,eq93,eq94,eq95,eq96}). Over the resulting candidate set in $\phi$, the smallest value recorded on the $(\alpha,\beta)$ grid yielded \cref{eq97}, attained as stated at $(\alpha,\beta,\phi)=(0,1,\phi_0)$. The exact value shown comes from \cref{eq93}.

Again, as an independent check, we used the Maple routine \texttt{Minimize} over the box $\alpha, \beta \in [0,1]$, $\phi \in [\phi_0,\pi]$, with Digits set very high, to verify \cref{eq97}. Finally, in \cite{Dunster:2026:BWM} surface plots of $P_5(\alpha,\beta,\phi)$ for $\alpha, \beta \in [0,1]$, as well as for $(\alpha, \beta)$ near $(0,1)$, animated over $\phi_0 \le \phi \le \pi$, provide further visual confirmation. Maple source code for the computation of \cref{eq97} is also given.

\section{Proof of \texorpdfstring{\cref{thm:main}}{} for \texorpdfstring{$n=7,9,11,13,\ldots$}{}}
\label{sec:n>5}

Define
\begin{equation}
\label{eq98}
\nu = \alpha+\beta \in(0,2),
\end{equation}
\begin{equation}
\label{eq99}
c_1(\alpha,\beta)
= \tfrac12 (\alpha-\beta)(1-\beta),
\end{equation}
\begin{equation}
\label{eq100}
c_2(\alpha,\beta)
= \tfrac12 (\alpha-\beta)(1+\alpha),
\end{equation}
and
\begin{equation}
\label{eq101}
J(\alpha,\beta,\phi;n)
= n^{1-\alpha}P_n(\alpha,\beta,\phi),
\end{equation}
so that, using \cref{eq90},
\begin{multline}
\label{eq102}
J(\alpha,\beta,\phi;n)
= \nu^{-1}A_1(\alpha,\phi)
\left\{n+c_1(\alpha,\beta) \right\}
\\
- \nu^{-1}n^{-\nu}A_2(\alpha,\beta,\phi)
\left\{n+c_2(\alpha,\beta)\right\}
- |m_0|\,\Gamma(2-\beta)
- |m_\infty|\,g_{\alpha,\beta}(n),
\end{multline}
where
\begin{equation}
\label{eq103}
g_{\alpha,\beta}(n)
= \Gamma(3+\alpha,n)\,n^{-1-\nu}.
\end{equation}

The purpose of this section is to prove the following, which in conjunction with \cref{eq89,eq97,eq101}, proves \cref{thm:main} for $n=7,9,11,\ldots$.

\begin{theorem}
\label{thm:Jn}
For $n \ge 5$ (not necessarily integral), $\alpha,\beta \in (0,1]$ and $\phi \in [\phi_{0},\pi)$
\begin{equation}
\label{eq104}
J(\alpha,\beta,\phi;n) \ge J(\alpha,\beta,\phi;5).
\end{equation}
\end{theorem}

From \cref{prop:Pn-Bound} and \cref{eq101} it suffices to show that $\partial J/ \partial n  \ge 0$ for all $n\ge5$, where $n$ is temporarily regarded as a continuous real variable rather than an integer. To this end, on differentiating \cref{eq102} with respect to $n$, we find
\begin{multline}
\label{eq105}
\partial J(\alpha,\beta,\phi;n) /\partial n
= \nu^{-1}A_1(\alpha,\phi)
- \nu^{-1}(1-\nu)A_2(\alpha,\beta,\phi)\,n^{-\nu} \\
+ A_2(\alpha,\beta,\phi)\,c_2(\alpha,\beta)\,n^{-1-\nu}
- |m_\infty|\,g_{\alpha,\beta}'(n).
\end{multline}
Thus $\partial J/ \partial n  \sim \nu^{-1}A_1>0$ for large $n$ and $\nu>0$, which indicates, at least asymptotically, the desired positivity, and hence that the desired bound \cref{eq104} holds. We now proceed to confirm this.

To do so, we note from \cref{eq100} that $c_2(\alpha,\beta)$, which appears in the third term on the RHS of \cref{eq105}, can take either sign, and so we split the analysis into the two cases $c_2(\alpha,\beta)\ge0$ and $c_2(\alpha,\beta)<0$.

\subsection{Case 1: \texorpdfstring{$c_2(\alpha,\beta)\ge 0 \; (\alpha\ge\beta)$}{}}

Firstly, from \cref{eq103} and using $\partial \Gamma(a,x)/\partial x = -x^{-1+a}e^{-x}$ yields
\begin{equation}
\label{eq106}
g_{\alpha,\beta}'(n)
= -(1+\nu)\,n^{-2-\nu}\,\Gamma(3+\alpha,n)
- n^{1-\beta}e^{-n} < 0,
\end{equation}
for $n>0$ and $\alpha,\beta \in (0,1]$.

Now for the present case, the $c_2(\alpha,\beta)$ term in \cref{eq105} is nonnegative, and by \cref{eq106} the last term is also nonnegative. Consequently, noticing that $n^{-\nu}\le5^{-\nu}$ for $n\ge5$, we have
\begin{equation}
\label{eq107}
\partial J(\alpha,\beta,\phi;n) /\partial n 
\ge  F_1(\alpha,\beta,\phi)
\quad  (n \ge 5, \, 0 \le \beta \le \alpha \le 1),
\end{equation}
where
\begin{equation}
\label{eq108}
F_1(\alpha,\beta,\phi)
= \frac{1}{\nu}\left[
A_1(\alpha,\phi)
-\frac{\max\{1-\nu,0\}}{5^{\nu}}A_2(\alpha,\beta,\phi)
\right].
\end{equation}
Since $A_1(\alpha,\phi)$ and $A_2(\alpha,\beta,\phi)$ are smooth in $(\alpha,\beta,\phi)$ and from \cref{eq63,eq66}
\begin{equation*}
A_1(\alpha,\phi)-A_2(\alpha,\beta,\phi)(1-\nu)
5^{-\nu}=\mathcal{O}(\nu)
\quad (\nu \to 0),
\end{equation*}
the quantity $F_1(\alpha,\beta,\phi)$ is
continuous and bounded on the parameter regime under consideration. Incidentally, since $A_1(\alpha,\phi) > 0$ it is certainly clear that $F_1(\alpha,\beta,\phi)>0$ for $\nu=\alpha+\beta \ge 1$.

\subsection{Case 2: \texorpdfstring{$c_2(\alpha,\beta)<0 \; (\alpha<\beta)$}{}}

For $n\ge5$ we have
\begin{equation*}
-A_2(\alpha,\beta,\phi)(1-\nu)\,n^{-\nu}
\ge -A_2(\alpha,\beta,\phi) |1-\nu|\,5^{-\nu},
\end{equation*}
and
\begin{equation*}
\nu A_2(\alpha,\beta,\phi) c_2(\alpha,\beta)\,n^{-1-\nu}
\ge -\nu A_2(\alpha,\beta,\phi)
|c_2(\alpha,\beta)|\,5^{-1-\nu}.
\end{equation*}
Therefore, in conjunction with \cref{eq105,eq106}, it is evident that
\begin{equation}
\label{eq112}
\partial J(\alpha,\beta,\phi;n) /\partial n
\ge F_2(\alpha,\beta,\phi)
\quad  (n \ge 5, \, 0 \le \alpha < \beta \le 1),
\end{equation}
where
\begin{equation}
\label{eq113}
F_2(\alpha,\beta,\phi)
= \frac{1}{\nu}\left[
A_1(\alpha,\phi)
- \frac{1}{5^{\nu}}\left\{|1-\nu|
+ \frac{\nu |c_2(\alpha,\beta)|}{5}
\right\}A_2(\alpha,\beta,\phi)
\right].
\end{equation}

Thus, in summary, to establish $\partial J /\partial n \ge 0$ for $n \ge 5$ it is sufficient from \cref{eq107,eq112} to check that $F_1(\alpha,\beta,\phi)\ge 0$ for $0 \le \beta \le \alpha \le 1$, and $F_2(\alpha,\beta,\phi)\ge 0$ for $0 \le \alpha<\beta \le 1$. The following then completes the proof of \cref{thm:Jn}.

\begin{proposition}
\label{prop:F1F2}
For $\phi_0=0.061$ we have
\begin{equation}
\label{eq115}
\inf_{\substack{\alpha,\beta\in(0,1)
,\, \alpha \ge \beta
\\ \phi \in [\phi_0,\pi)}} 
F_1(\alpha,\beta,\phi)
=0.03251857515\cdots,
\end{equation}
attained at $(\alpha,\beta,\phi)=(1,1,\phi_0)$, and
\begin{equation}
\label{eq116}
\inf_{\substack{\alpha,\beta\in(0,1)
,\, \alpha < \beta
\\ \phi \in [\phi_0,\pi)}} 
F_2(\alpha,\beta,\phi)
=0.03204047407\cdots,
\end{equation}
attained at $(\alpha,\beta,\phi)= (\alpha_2,1,\phi_0)$, where $\alpha_2=0.8583872779\cdots$. In particular, both infima are strictly positive.
\end{proposition}

The computation of these values was performed in a similar manner to the method used for \cref{prop:Pn-Bound}, that is, sampling over a fine $\alpha, \beta$ grid, and $\phi$ only at boundary points ($\phi_0$ and $\pi$) and at any interior critical points. The infima were then independently confirmed using the Maple \texttt{Optimization[Minimize]} package. The grid and optimization code can be found in \cite{Dunster:2026:BWM}.

For the boundary points we again utilised exact limits, which in this case are given by
\begin{equation}
\label{eq117}
F_1(0,0,\phi)=F_2(0,0,\phi)
=\frac{\ln\left\{\csc(\tfrac{1}{2}\phi)\right\}
\left[ 2 \{1+\gamma+\ln(10)\}
+\ln\left\{\sin(\tfrac{1}{2}\phi)\right\}\right]}
{2\pi\,\Delta^2},
\end{equation}
\begin{equation}
\label{eq118}
F_1(1,1,\phi)=\frac{1-\sin\left(\tfrac12 \phi\right)}
{\pi(\pi-\phi)^2},
\end{equation}
and
\begin{equation}
\label{eq119}
F_2(\alpha,1,\phi)
=\frac{1}{1+\alpha}\left[
A_1(\alpha,\phi)
-\frac{(1-\alpha)\,C_2(1,\phi)}
{5^{1+\alpha}\,\pi}
\left\{\alpha
+\frac{1}{10}
(1+\alpha)^2(1-\alpha)\right\} \right],
\end{equation}
where from \cref{eq25,eq33}
\begin{equation}
\label{eq120}
C_2(1,\phi)=\frac{\csc\left(\tfrac12 \phi\right)-1}
{2(\pi-\phi)^2}.
\end{equation}

In addition
\begin{equation}
\label{eq121}
F_1(\alpha,\beta,\pi)
=\frac{1}{\nu}\left[
A_1(\alpha,\pi)
-\frac{\max\{1-\nu,0\}A_2(\alpha,\beta,\pi)}{5^{\nu}}
\right],
\end{equation}
\begin{equation}
\label{eq122}
F_2(\alpha,\beta,\pi)
=
\frac{1}{\nu}\left[
A_1(\alpha,\pi)-\frac{A_2(\alpha,\beta,\pi)}{5^{\nu}}\left(
|1-\nu|
+\frac{\nu\left|c_2(\alpha,\beta)\right|}{5}
\right)\right],
\end{equation}
where $A_1(\alpha,\pi)$, $A_2(\alpha,\beta,\pi)$ and $c_2(\alpha,\beta)$ are given by \cref{eq65,eq68,eq100}.

Incidentally, in comparison to \cref{eq115,eq116}, from \cref{eq121,eq122} we find
\begin{equation*}
\inf_{\alpha \ge \beta} F_1(\alpha,\beta,\pi)
= F_1(1,1,\pi) = \frac{1}{8\pi}
>  \inf_{\alpha<\beta} F_2(\alpha,\beta,\pi)
=0.03287693288\cdots,
\end{equation*}
the latter being attained at $\alpha= 0.4852393602\cdots$ and $\beta=1$.

\begin{figure}
 \centering
 \includegraphics[
 width=0.9\textwidth,keepaspectratio]{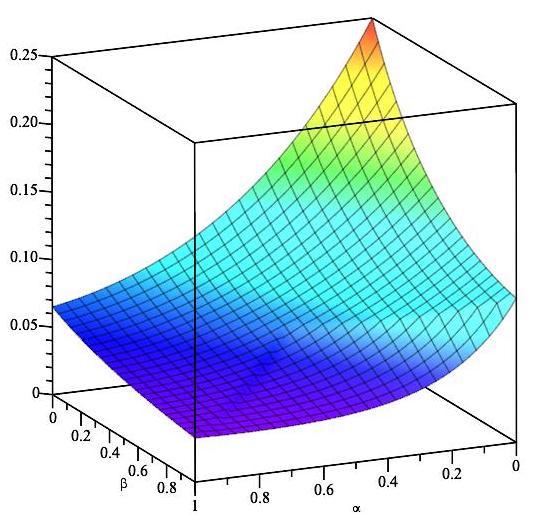}
 \caption{Graph of $F(\alpha,\beta,\phi_0)$ for $\alpha,\beta \in [0,1]$.}
 \label{fig:F1F2}
\end{figure}

In \cref{fig:F1F2} the surface of $F(\alpha,\beta,\phi_0)$ is shown, where
\begin{equation}
\label{eq124}
F(\alpha,\beta,\phi)
=\begin{cases}
F_1(\alpha,\beta,\phi)
& (\alpha\ge \beta)\\[1ex]
F_2(\alpha,\beta,\phi) & (\alpha<\beta)
\end{cases}\;,
\end{equation}
and this illustrates the infima of \cref{prop:F1F2}. For an animated version over $\phi \in [\phi_0,\pi]$ see \cite{Dunster:2026:BWM}.

\section{Extension to \texorpdfstring{$\phi \in [0,\phi_0)$}{} (\texorpdfstring{$\theta \in (\pi-\phi_0,\pi]$}{})}
\label{sec:Meijer}

The Watson large $n$ expansion breaks down near $\phi=0$, as can be seen from \cref{eq33} which becomes unbounded in this limit. In a subsequent paper we tackle $\phi \in [0,\phi_0)$ by replacing \cref{eq36,eq37} with uniform approximations for large $n$, valid near and at $\phi=0$, involving integrals of the form
\begin{equation}
\label{eq125}
M_{\lambda,\mu}(\zeta)
:=
\int_0^\infty t^\lambda (t^2+\zeta)^\mu e^{-t}\,dt,
\end{equation}
where $\zeta=n^2\rho^2$ and
\begin{equation}
\label{eq126}
\rho:=2\sin\left(\tfrac12\phi\right)
\sim \phi
\quad (\phi\to0).
\end{equation}
These integrals can be expressed in terms of Meijer $G$ functions; see \cite[Sec.~16.17]{NIST:DLMF}.

To see how integrals of the form \cref{eq125} apply, consider, for example, the troublesome second part of the integrand in \cref{eq28}, namely
\begin{equation}
\label{eq127}
\int_{0}^{\infty} (e^{s}-1)^{\alpha}
\left|e^{s}-e^{-i\phi}\right|^{-\beta}e^{-ns}\,ds,
\end{equation}
where $x=e^{i\phi}$. For small $s$ and $\phi$ one can show that the uniform first approximation for the integrand is given by
\begin{multline}
\label{eq128}
(e^{s}-1)^{\alpha}
\left|e^{s}-e^{-i\phi}\right|^{-\beta}
=(e^{s}-1)^{\alpha}
\left\{(e^s-1)^2+\rho^2e^s\right\}^{-\beta/2}
\\
\sim
s^\alpha(s^2+\rho^2)^{-\beta/2}
\quad (s\to0,\ \rho\to0).
\end{multline}
Substituting this into \cref{eq127} gives the formal approximation
\begin{equation}
\label{eq129}
\int_{0}^{\infty} (e^{s}-1)^{\alpha}
\left|e^{s}-e^{-i\phi}\right|^{-\beta}e^{-ns}\,ds
\sim
\int_0^\infty s^\alpha(s^2+\rho^2)^{-\beta/2}e^{-ns}\,ds.
\end{equation}
On making the change of variable $t=ns$, the right-hand side becomes
\begin{equation}
\label{eq130}
n^{-1-\alpha+\beta}
\int_0^\infty t^\alpha(t^2+n^2\rho^2)^{-\beta/2}e^{-t}\,dt
=
n^{-1-\alpha+\beta}M_{\alpha,-\beta/2}(\zeta),
\end{equation}
where $M_{\lambda,\mu}(\zeta)$ is given by \cref{eq125}, and remains bounded as $\zeta\to0^+$ ($\phi \to 0+$) whenever $\lambda+2\mu>-1$.

We remark that the argument $\zeta=n^2\rho^2\sim n^2\phi^2$  captures the transition $\phi\to0$, $n\to\infty$ with $n\phi$ bounded, which is precisely the uniform behaviour missing from the ordinary Watson expansion. For details on similar techniques, see \cite[Chap.~28]{Temme:2015:AMF}.

We finish by noting the following simple result, which confirms without using Meijer $G$ functions that \cref{eq05} holds ($n$ even or odd)  at $\theta=\pi$ ($\phi=0$), provided $\beta \ge \alpha$.

\begin{lemma}
\label{lem:phi=0}
Let $0<\alpha\le\beta<1$ and let $1 \le n\in\mathbb{N}$. Then
\begin{equation}
\label{eq131}
A_n(\alpha,\beta,1)>\bigl|A_n(\alpha,\beta,-1)\bigr|.
\end{equation}
\end{lemma}

\begin{proof}
For $\omega=-1$ it is straightforward to show from \cref{eq01} that
\begin{equation}
\label{eq132}
A_n(\alpha,\beta,-1)
=\frac{\Gamma(n-\alpha+\beta)}
{\Gamma(\beta-\alpha)\,n!}\ge0
\quad (0<\alpha\le\beta<1),
\end{equation}
so $\bigl|A_n(\alpha,\beta,-1)\bigr|=A_n(\alpha,\beta,-1)$ under the hypothesis of the lemma.

Next, again from the generating function \cref{eq01}, expanding $(1+\omega z)^{\alpha}$ and $(1-z)^{-\beta}$ and collecting
coefficients gives
\begin{equation}
\label{eq133}
A_n(\alpha,\beta,\omega)
=\sum_{k=0}^n \binom{\alpha}{k}\,
\frac{\Gamma(n-k+\beta)}{\Gamma(\beta)
\,\Gamma(n-k+1)}\,\omega^k,
\end{equation}
where $\binom{\alpha}{0}=1$ and, from \cref{eq08},
\begin{multline}
\label{eq134}
\binom{\alpha}{k}=\frac{\Gamma(1+\alpha)}
{\Gamma(1-k+\alpha)\,\Gamma(k+1)}
\\
=(-1)^{k+1}\frac{\Gamma(1+\alpha)\,\Gamma(k-\alpha)
\sin(\pi\alpha)}{\pi k!}
\quad (k=1,2,3,\ldots n).
\end{multline}
Thus for $\alpha \in (0,1)$ the $k=0$ and the $k$-odd coefficients in \cref{eq133} are positive. Therefore, using \cref{eq133} with $\omega=\pm1$, we obtain for $0<\alpha\le\beta<1$
\begin{multline*}
A_n(\alpha,\beta,1)-\left|A_n(\alpha,\beta,-1)\right|
=A_n(\alpha,\beta,1)-A_n(\alpha,\beta,-1)
\\
=\sum_{k=0}^n \binom{\alpha}{k}\,
\frac{\Gamma(n-k+\beta)}{\Gamma(\beta)\,\Gamma(n-k+1)}
\left\{1-(-1)^k\right\}
=2\sum_{\substack{1\le k\le n\\ k\ \mathrm{odd}}}\!\!
\binom{\alpha}{k}\,\frac{\Gamma(n-k+\beta)}{\Gamma(\beta)\,\Gamma(n-k+1)} >0.
\end{multline*}
\end{proof}

\section*{Acknowledgement}
Financial support from Ministerio de Ciencia e Innovación pro\-ject PID2024-159583NB-I00 (MICIU/ AEI / 10.13039/501100011033 / FEDER, UE) is acknowledged.

\section*{Conflict of interest}
The author declares no conflicts of interest.

\makeatletter
\interlinepenalty=10000

\bibliographystyle{siamplain}
\bibliography{biblio}

@misc{NIST:DLMF,
	howpublished = {Release 1.1.6 of 2022-06-30},
	key = {{\relax DLMF}},
	note = {F.~W.~J. Olver, A.~B. {Olde Daalhuis}, D.~W. Lozier, B.~I. Schneider, R.~F. Boisvert, C.~W. Clark, B.~R. Miller, B.~V. Saunders, H.~S. Cohl, and M.~A. McClain, eds.},
	title = {{\it NIST Digital Library of Mathematical Functions}},
	url = {http://dlmf.nist.gov/}}

@book{Wong:1989:AAI,
	author = {Wong, R.},
	title = {Asymptotic approximations of integrals},
    publisher = {Academic Press Inc., Boston-New York},
	year = {1989}}

@book{Temme:2015:AMF,
	Author = {Temme, Nico M.},
	Isbn = {978-981-4612-15-9},
	Mrclass = {41-02 (33Cxx 33E20 65D30)},
	Mrnumber = {3328507},
	Mrreviewer = {Jos\~A\copyright{} Luis L\~A${}^3$pez},
	Pages = {xxii+605},
	Publisher = {World Scientific Publishing Co. Pte. Ltd., Hackensack, NJ},
	Series = {Series in Analysis},
	Title = {Asymptotic methods for integrals},
	Volume = {6},
	Year = {2015}}

@book{Olver:1997:ASF,
	Address = {Wellesley, MA},
	Author = {Olver, F. W. J.},
	Isbn = {1-56881-069-5},
	Mrclass = {41-02 (33Cxx 41A60 65D20)},
	Mrnumber = {MR1429619 (97i:41001)},
	Note = {Reprint of the 1974 original [Academic Press, New York]},
	Pages = {xviii+572},
	Publisher = {A K Peters Ltd.},
	Series = {AKP Classics},
	Title = {Asymptotics and special functions},
	Year = {1997}}

@article{Aharonov:1972:OAI,
  Author  = {Aharonov, D. and Friedland, S.},
  Title   = {On an inequality connected with the coefficient conjecture for functions of bounded boundary rotation},
  Journal = {Ann. Acad. Sci. Fenn. Ser. A I},
  Number  = {524},
  Year    = {1972},
  Pages   = {1-14},
  DOI     = {10.5186/aasfm.1973.524}
}

@article{Barnard:2015:BCA,
  Author  = {Barnard, R. W. and Jayatilake, Udaya C. and Solynin, Alexander Yu.},
  Title   = {Brannan's conjecture and trigonometric sums},
  Journal = {Proc. Amer. Math. Soc.},
  Volume  = {143},
  Number  = {5},
  Pages   = {2117-2128},
  Year    = {2015},
  DOI     = {10.1090/S0002-9939-2015-12398-2}
}

@article{Barnard:1997:OAC,
  Author  = {Barnard, Roger W. and Pearce, Kent and Wheeler, William},
  Title   = {On a coefficient conjecture of {B}rannan},
  Journal = {Complex Var. Theory Appl.},
  Volume  = {33},
  Number  = {1-4},
  Pages   = {51-61},
  Year    = {1997},
  DOI     = {10.1080/17476939708815011}
}

@article{Barnard:2021:ADP,
  Author  = {Barnard, Roger W. and Richards, Kendall C.},
  Title   = {A direct proof of {B}rannan's conjecture for $\beta=1$},
  Journal = {J. Math. Anal. Appl.},
  Volume  = {493},
  Number  = {2},
  Pages   = {124534},
  Year    = {2021},
  DOI     = {10.1016/j.jmaa.2020.124534}
}

@incollection{Brannan:1974:OCP,
  Author    = {Brannan, D. A.},
  Title     = {On coefficient problems for certain power series},
  Booktitle = {Proceedings of the Symposium on Complex Analysis (Univ. Kent, Canterbury, 1973)},
  Pages     = {17-27},
  Publisher = {Cambridge Univ. Press},
  Address   = {London},
  Series    = {London Math. Soc. Lecture Note Ser.},
  Number    = {12},
  Year      = {1974}
}

@article{Deniz:2020:TFS,
  Author  = {Deniz, E. and {\c{C}}a{\u{g}}lar, M. and Sz{\'a}sz, R.},
  Title   = {The final step in a proof of {B}rannan's conjecture for $\beta=1$},
  Journal = {J. Math. Anal. Appl.},
  Volume  = {487},
  Number  = {2},
  Pages   = {124001},
  Year    = {2020},
  DOI     = {10.1016/j.jmaa.2020.124001}
}

@article{Jayatilake:2013:BCF,
  Author  = {Jayatilake, Udaya C.},
  Title   = {Brannan's conjecture for initial coefficients},
  Journal = {Complex Var. Elliptic Equ.},
  Volume  = {58},
  Number  = {5},
  Pages   = {685-694},
  Year    = {2013},
  DOI     = {10.1080/17476933.2011.605445}
}

@article{Milcetich:1989:OAC,
  Author  = {Milcetich, John G.},
  Title   = {On a coefficient conjecture of {B}rannan},
  Journal = {J. Math. Anal. Appl.},
  Volume  = {139},
  Number  = {2},
  Pages   = {515-522},
  Year    = {1989},
  DOI     = {10.1016/0022-247X(89)90125-X}
}

@article{Ruscheweyh:2007:OBC,
  Author  = {Ruscheweyh, Stephan and Salinas, Luis},
  Title   = {On {B}rannan's coefficient conjecture and applications},
  Journal = {Glasg. Math. J.},
  Volume  = {49},
  Number  = {1},
  Pages   = {45-52},
  Year    = {2007},
  DOI     = {10.1017/S0017089507003400}
}

@article{Szasz:2020:OTB,
  Author  = {Sz{\'a}sz, R{\'o}bert},
  Title   = {On the {B}rannan's conjecture},
  Journal = {Mediterr. J. Math.},
  Volume  = {17},
  Number  = {1},
  Pages   = {38},
  Year    = {2020},
  DOI     = {10.1007/s00009-019-1469-9}
}

@article{Cotirla:2024:OTG,
  author  = {Lumini{\c{t}}a-Ioana Cot\^{\i}rl\u{a} and R{\'o}bert Sz{\'a}sz},
  title   = {On the general case of {B}rannan conjecture},
  journal = {J. Math. Inequal.},
  volume  = {18},
  number  = {3},
  year    = {2024},
  pages   = {953-969},
  doi     = {10.7153/jmi-2024-18-52},
  url     = {}
}

@misc{Dunster:2026:BWM,
  author       = {T. M. Dunster},
  title        = {{B}rannan-{W}atson: Maple source code and animated figures},
  howpublished = {\url{https://github.com/tmdunster/Brannan-Watson#files-and-animations}},
  note         = {GitHub repository, accessed 2026-02-15}
}
\end{document}